\documentclass[10pt]{article}

\usepackage{fullpage}
\usepackage{amsmath,amsthm,dsfont,amsfonts,amssymb,fancyhdr,mathtools}
\usepackage{graphicx}
\usepackage{color}
\usepackage{cases}
\usepackage{enumerate}
\usepackage[colorlinks=true, allcolors=blue]{hyperref}
\usepackage{verbatim}
\usepackage{tikz}
\usepackage{authblk}
\usepackage[normalem]{ulem}

\newcommand{\gauss}[2]{\genfrac{[}{]}{0pt}{}{#1}{#2}}

\theoremstyle{definition}
\newtheorem{definition}{Definition}[section]
\newtheorem{remark}[definition]{Remark}
\newtheorem{example}[definition]{Example}

\theoremstyle{plain}
\newtheorem{theorem}[definition]{Theorem}
\newtheorem{lemma}[definition]{Lemma}
\newtheorem{corollary}[definition]{Corollary}
\newtheorem{proposition}[definition]{Proposition}
\newtheorem{conjecture}[definition]{Conjecture}

\def\Ga{\Gamma}

\def\MxR{\operatorname{Mat}_X(\mathbb{R})}
\begin{document}
\title{\bf On the (non-)existence of tight distance-regular\\ 
graphs: a local approach}
%A local approach to tight distance-regular graphs
%A conjecture on tight distance-regular graphs

\author[a,b]{Jack H. Koolen}
\author[c]{Jae-Ho Lee${}^*$}
\author[d]{Shuang-Dong Li}
\author[e]{Yun-Han Li}
\author[e,f]{\\Xiaoye Liang${}^*$}
\author[e,f]{Ying-Ying Tan}

\affil[a]{\small School of Mathematical Sciences, University of Science and Technology of China, Hefei, Anhui, 230026, PR China}
\affil[b]{\small CAS Wu Wen-Tsun Key Laboratory of Mathematics, University of Science and Technology of China, 96 Jinzhai Road, Hefei, Anhui, 230026, PR China}
\affil[c]{\small Department of Mathematics and Statistics, University of North Florida, Jacksonville, FL 32224, U.S.A}
\affil[d]{\small Jianghuai College, Anhui University, Hefei, Anhui,   230031, PR China}
\affil[e]{\small School of Mathematics and Physics,  Anhui Jianzhu University, Hefei, Anhui,  230601, PR China}
\affil[f]{\small Operations Research and Data Science Laboratory, Anhui Jianzhu University, Hefei, Anhui,   230601, PR China}

\date{}
%\date{(Updated v.5: November 29, 2023)}
\maketitle
\newcommand\blfootnote[1]{%
\begingroup
\renewcommand\thefootnote{}\footnote{#1}%
\addtocounter{footnote}{-1}%
\endgroup}
%\blfootnote{2010 Mathematics Subject Classification. 05C50, 05E30.}
\blfootnote{\hspace{-0.5em}${}^*$Corresponding authors}
\blfootnote{Email addresses: koolen@ustc.edu.cn (J.H. Koolen), jaeho.lee@unf.edu (J.-H. Lee), lisd@ahu.edu.cn (S.-D. Li),\\ liyunhan1998@163.com (Y.-H. Li), liangxy0105@foxmail.com (X. Liang), tansusan1@ahjzu.edu.cn (Y.-Y. Tan)}

\vspace{-1cm}
\begin{abstract}
\noindent
Let $\Gamma$ denote a distance-regular graph with diameter $D\geq 3$.
Juri\v{s}i{\'c} and Vidali conjectured that if $\Gamma$ is tight with classical parameters $(D,b,\alpha,\beta)$, $b\geq 2$, then $\Gamma$ is not locally the block graph of an orthogonal array nor the block graph of a Steiner system.
In the present paper, we prove this conjecture and, furthermore, extend it from the following aspect.
Assume that for every triple of vertices $x, y, z$ of $\Gamma$, where $x$ and $y$ are adjacent, and $z$ is at distance $2$ from both $x$ and $y$, the number of common neighbors of $x$, $y$, $z$ is constant.
We then show that if $\Gamma$ is locally the block graph of an orthogonal array (resp.~a Steiner system) with smallest eigenvalue $-m$, $m\geq 3$, then the intersection number $c_2$ is not equal to $m^2$ (resp. $m(m+1)$).
Using this result, we prove that if a tight distance-regular graph $\Gamma$ is not locally the block graph of an orthogonal array or a Steiner system, then the valency (and hence diameter) of $\Gamma$ is bounded by a function in the parameter $b=b_1/(1+\theta_1)$, where $b_1$ is the intersection number of $\Gamma$ and $\theta_1$ is the second largest eigenvalue of $\Gamma$.


%\hfil\break
\end{abstract}

\section{Introduction}\label{Sec:Intro}
Let $\Gamma$ denote a distance-regular graph with diameter $D\geq 3$, intersection numbers $a_i, b_i, c_i$ $(0\leq i \leq D)$, and eigenvalues $k=\theta_0>\theta_1> \cdots > \theta_D$.
Juri{\v s}i{\'c}, Koolen, and Terwilliger \cite{AJJKPT} showed that $\Gamma$ satisfies the following inequality:
\begin{equation}\label{eq:tightDRG}
	\left( \theta_1 + \frac{k}{a_1+1}\right)\left(\theta_D + \frac{k}{a_1+1}\right) \geq -\frac{ka_1b_1}{(a_1+1)^2}.
\end{equation}
We say $\Gamma$ is \emph{tight} whenever $\Gamma$ is nonbipartite and equality holds in \eqref{eq:tightDRG}.
Tight distance-regular graphs have been studied with considerable attention and characterized in various ways; see \cite{vDKT2016,GoTer2002,Pasc2001,Pasc2001LAA}.
A notable characterization is that, for each vertex $x$ in a tight distance-regular graph, its local graph at $x$ is a connected strongly regular graph with eigenvalues 
\begin{equation}\label{eq:local eig tight}
a_1, \qquad r:=-1-\frac{b_1}{1+\theta_D}, \qquad s:=-1-\frac{b_1}{1+\theta_1},
\end{equation}
see \cite[Theorem 12.6]{AJJKPT}.
Suppose that $\Gamma$ is tight with $D \geq 3$, and let $\Delta$ denote a local graph of $\Gamma$. 
We observe that $\Delta$ is a connected strongly regular graph with eigenvalues $a_1, r, s$. 
Throughout this paper, we assume that $r$ and $s$ are integers.  
Because if they are not, $\Delta$ is a conference graph, which implies that $\Gamma$ is a Taylor graph; see \cite{KAGL2024+, KP2012}.
%Because if they are not, then $\Delta$ is a conference graph, and thus $\Gamma$ satisfies $a_1 = k/2-1$. 
%According to the classification in \cite[Theorem 16]{KP2012}, $\Gamma$ is then either a line graph, a Taylor graph, the Johnson graph $J(7,3)$, or the halved $7$-cube.
Therefore, further discussion of $\Gamma$ in this paper is unnecessary when $r$ and $s$ are not integers.

\smallskip
\noindent
Suppose that $s \leq -2$, that is, the smallest eigenvalue of $\Delta$ is less than or equal to $-2$.
For notational convenience, we set $m:=-s$ and $n:=r-s$.
By Sims' result (cf. \cite[Theorem 5.1]{Neumaier1979}), $\Delta$ belongs to one of the following families: (i) complete multipartite graphs with classes of size $m$, (ii) block graphs of orthogonal arrays $\operatorname{OA}(m,n)$, (iii) block graphs of Steiner systems $S(2, m, mn+m-n)$, (iv) finitely many further graphs.
If $\Gamma$ has classical parameters $(D,b,\alpha,\beta)$, then in case (i), $\Gamma$ is the complete multipartite graph $K_{(n+1),m}$ with $D=2$ \cite[Proposition 1.1.5]{bcn89}. 
For cases (ii) and (iii), when $\Gamma$ has diameter $D=3$, we must have $b=1$. 
This restriction implies that $\Gamma$ is one of the following three graphs: the Johnson graph $J(6,3)$, the halved $6$-cube, or the Gosset graph $E_7(1)$; see \cite[Section 7]{AJJV}.
Hence, our focus lies on cases where $D\geq 4$ and $b\geq 2$.
Juri\v{s}i\'{c} and Vidali posed the following conjecture:

\begin{conjecture}[{\cite[Conjecture 2]{AJJV}}]\label{conj:JV} 
Let $\Gamma$ be a tight distance-regular graph with classical parameters $(D,b,\alpha, \beta)$, $b\geq 2$, and diameter $D\geq 4$.
For a vertex $u$ of $\Gamma$, the local graph of $\Gamma$ at $u$ is not the block graph of an orthogonal array or a Steiner system.
\end{conjecture}

In the present paper, we prove this conjecture and extend it to the case where a tight distance-regular graph $\Gamma$ has no classical parameters; see Theorem \ref{main thm2} and Corollary \ref{cor:sub-main}.
Furthermore, we extend the conjecture from the following viewpoint.
Let $\Gamma$ be a distance-regular graph with diameter $D\geq 3$.
Note that a tight distance-regular graph is $1$-homogeneous in the sense of Nomura \cite[Theorem 11.7]{AJJKPT}.
We consider a regular property for $\Gamma$ that is a more general concept than the $1$-homogeneous property: we say the (\emph{triple}) \emph{intersection number $\gamma(\Gamma)$ exists} if, for every triple of vertices $(x,y,z)$ of $\Gamma$ such that $x$ and $y$ are adjacent and $z$ is at distance $2$ from both $x$ and $y$, the number of common neighbors of $x$, $y$, and $z$ is constant and equal to $\gamma(\Gamma)$.
To avoid the degenerate case, we assume that there exists at least one such triple $(x,y,z)$ in $\Gamma$ (i.e., $a_2\neq 0$) when we say $\gamma(\Gamma)$ exists.
The result of our extension is the main result of this paper and is as follows:
\begin{theorem}\label{main thm1}
Let $\Gamma$ be a distance-regular graph with diameter $D\geq 3$, valency $k$, and intersection number $c_2$. 
Assume that $\Gamma$ is locally strongly regular with smallest eigenvalue $-m$, where $m\geq 3$, and the intersection number $\gamma(\Gamma)$ exists.
Then the following {\rm(i)} and {\rm(ii)} hold.
\begin{enumerate}[\normalfont(i)]
	\item If $\Gamma$ is locally the block graph of an orthogonal array and $k > m^2$, then $c_2 \neq m^2$.
	\item If $\Gamma$ is locally the block graph of a Steiner system and $k > m(m+1)$, then $c_2 \neq m(m+1)$.
\end{enumerate}
\end{theorem}

\noindent
Theorem \ref{main thm1} is relevant to the problem of determining an upper bound on the diameter of a tight distance-regular graph.
In the theory of distance-regular graphs, establishing an upper bound for the diameter of distance-regular graphs in terms of some intersection numbers is an important problem.
In particular, with respect to the valency $k=b_0$, various bounds for the diameter have been known and have contributed to the theory of distance-regular graphs; see \cite{MillerSiran2013}. 
One of the significant results of these contributions is the proof of the Bannai-Ito conjecture \cite[p.~237]{BI1984} by Bang, Dubickas, Koolen, and Moulton \cite{BDKM2015}.

\medskip
\noindent\textbf{Bannai-Ito Conjecture.}
\textit{There are finitely many distance-regular graphs with fixed valency at least three.}

\medskip
\noindent
To prove this conjecture, they demonstrated that the diameter of the distance-regular graph is bounded by a univariate function with the variable valency $k$.
Returning our attention to the present paper, we will discuss an upper bound on the diameter in a tight distance-regular graph using a specific parameter, distinct from valency $k$. 
Specifically, by utilizing the result of Theorem \ref{main thm1}, we will show that when a tight distance-regular graph is not locally the block graph of an orthogonal array or a Steiner system, its diameter is bounded by a function of the parameter $b=b_1/(1+\theta_1)$. 
We present this finding in the following theorem.
\begin{theorem}\label{main thm3}
Let $\Gamma$ be a tight distance-regular graph with diameter $D\geq 3$, intersection number $b_1$, and eigenvalues $k>\theta_1>\cdots> \theta_D$.
Define 
$$
	b:=b_1/(1+\theta_1).
$$
We assume $b\geq 2$.
If a local graph of $\Gamma$ is neither the block graph of an orthogonal array nor the block graph of a Steiner system, then the valency $k$ (and hence diameter $D$) of $\Gamma$ is bounded by a function of $b$.
\end{theorem}

\noindent
In Remark \ref{rmk:bound for k}, we give an explicit bound in terms of $b$ for the valency of $\Gamma$. 
From Theorem \ref{main thm3}, it follows that the diameter of a tight distance-regular graph with classical parameters $(D,b,\alpha,\beta)$, $D\geq 3$, and $b\geq 2$, is bounded by a function of $b$; see Corollary \ref{cor:t-DRG,D,b}. 

\smallskip
This paper is organized as follows.
In Section \ref{Sec:Prelim}, we present basic definitions and some known results about distance-regular graphs.
Section \ref{Sec:LS(m,n)} discusses the block graph of an orthogonal array and its properties.
We then analyze the structure of the $\mu$-graph of an amply regular graph that is locally the block graph of an orthogonal array.
Following that, Section \ref{Sec:Steiner S(m,n)} covers the block graph of a Steiner system and its properties. 
We also analyze the structure of the $\mu$-graph of an amply regular graph that is locally the block graph of a Steiner system.
In Section \ref{Sec:main thm1}, we revisit results related to the triple intersection number of a distance-regular graph and dedicate this section to proving our main result, Theorem \ref{main thm1}.
We conclude this section with a discussion of the case of tight distance-regular graphs with diameter three.
Section \ref{Sec:conjecture} provides the proof of Conjecture \ref{conj:JV} using Theorem \ref{main thm1}.
Finally, the paper concludes in Section \ref{Sec:main thm3} with the proof of Theorem \ref{main thm3} and a discussion of further direction.

%%%%%%%%%%%%%%%%%%%%%%%%%%%%%%%%%%%%%%%%%%%%%%
%%%%%%%%%%%%%%%%%%%%%%%%%%%%%%%%%%%%%%%%%%%%%%
%%%%%%%%%%%%%%%%%%%%%%%%%%%%%%%%%%%%%%%%%%%%%%
%%%%%%%%%%%%%%%%%%%%%%%%%%%%%%%%%%%%%%%%%%%%%%
%%%%%%%%%%%%%%%%%%%%%%%%%%%%%%%%%%%%%%%%%%%%%%
%%%%%%%%%%%%%%%%%%%%%%%%%%%%%%%%%%%%%%%%%%%%%%
\section{Preliminaries}\label{Sec:Prelim}

In this section, we review the basic definitions and some known results concerning distance-regular graphs that we will use later. 
For more background information, refer to \cite{bcn89}.

\smallskip
Throughout this section, let $\Gamma$ denote a finite, undirected, connected, and simple graph.
We denote $V(\Gamma)$ by the vertex set of $\Gamma$.
For vertices $x,y \in V(\Gamma)$, the \emph{distance} between $x$ and $y$, denoted as $\partial(x,y)$,  is the length of a shortest path from $x$ to $y$ in $\Gamma$.
The \emph{diameter} $D$ of $\Gamma$ is the maximum value of $\partial(x,y)$ for all pairs of vertices $x$ and $y$ of  $\Gamma$.
Suppose that $\Gamma$ has diameter $D$.
For $x\in V(\Gamma)$ and an integer $0\leq i \leq D$, define $\Gamma_i(x) = \{y \in V(\Gamma)\mid \partial(x,y) = i\}$.
Abbreviate $\Gamma(x)=\Gamma_1(x)$.
For an integer $k\geq 0$ we say $\Gamma$ is \emph{regular with valency $k$} (or \emph{$k$-regular}) if $|\Gamma(x)|=k$ for every $x\in V(\Gamma)$.

\smallskip
We now recall some special regular graphs.
We say the graph $\Gamma$ is \emph{distance-regular} whenever for all integers $0\leq h,i,j \leq D$ and for all vertices $x,y \in V(\Gamma)$ with $\partial(x,y)=h$, the number $p^h_{i,j}=|\Gamma_i(x)\cap \Gamma_j(y)|$ is independent of $x$ and $y$.
The numbers $p^h_{i,j}$ are called the \emph{intersection numbers} of $\Gamma$.
By construction, we observe that $p^h_{i,j}=p^h_{j,i}$ for $0\leq i,j,h\leq D$.
We abbreviate 
\begin{equation*}
	c_i = p^i_{1,i-1}, \qquad 
	a_i = p^i_{1,i}, \qquad 
	b_i = p^i_{1,i+1}, \qquad \qquad 
	(0 \leq i \leq D).
\end{equation*}
Observe that $\Gamma$ is regular with valency $k=b_0$.
Moreover, we note that $a_0 = b_D = c_0 = 0$, $c_1=1$, and $a_i+b_i+c_i = k$ for $0\leq i \leq D$.
We refer to the sequence $\{b_0, b_1, \ldots, b_{D-1}; c_1, c_2, \ldots, c_D\}$ as the \emph{intersection array} of $\Gamma$.
Next, consider the following regularity properties of the graphs below:
\begin{itemize}\itemsep0em 
	\item[(i)] Every pair of adjacent vertices has precisely $\lambda$ common neighbors.
	\item[(ii)] Every pair of vertices at distance $2$ has precisely $\mu$ common neighbors. 
	\item[(iii)] Every pair of nonadjacent vertices has precisely $\mu$ common neighbors. 
\end{itemize}
Let $\Gamma$ be $\kappa$-regular with $\nu$ vertices.
We say $\Gamma$ is \emph{amply regular} with parameters $(\nu,\kappa,\lambda,\mu)$ if (i) and (ii) hold.
We also say $\Gamma$ is \emph{strongly regular} with parameters $(\nu,\kappa,\lambda,\mu)$ if (i) and (iii) hold.
Observe that every distance-regular graph is amply regular with $\lambda=a_1$ and $\mu=c_2$. Moreover, every distance-regular graph with $D\leq 2$ is strongly regular.
If $\Gamma$ is a connected strongly regular graph with parameters $(\nu, \kappa, \lambda, \mu)$ and diameter two, then it has precisely three distinct eigenvalues $\kappa > r>s$, satisfying
\begin{equation}\label{eq:srg eigval}
	\nu = \frac{(\kappa - r)(\kappa-s)}{\kappa+rs}, \qquad
	\lambda = \kappa + r + s + rs, \qquad
	\mu = \kappa + rs.
\end{equation}
The following is an example of a strongly regular graph for later use in this paper.
\begin{example}
A \emph{generalized quadrangle} is an incidence structure such that:
(i) every pair of points is on at most one line, and hence every pair of lines meets in at most one point,
(ii) if $p$ is a point not on a line $L$, then there is a unique point $p'$ on $L$ such that $p$ and $p'$ are collinear.
If every line contains $s+1$ points, and every point lies on $t+1$ lines, we say that the generalized quadrangle has order $(s,t)$, denoted by $\operatorname{GQ}(s,t)$.
The \emph{point graph} of a generalized quadrangle is the graph with the points of the quadrangle as its vertices, where two points are adjacent if and only if they are collinear. 
The point graph of a $\operatorname{GQ}(s,t)$ is strongly regular with parameters
\begin{equation*}
	\nu=(s+1)(st+1), \qquad \kappa=s(t+1), \qquad \lambda=s-1, \qquad \mu=t+1.
\end{equation*}
\end{example}

We recall the notion of a complete multipartite graph.
A \emph{clique} in $\Gamma$ is a subset of $V(\Gamma)$ such that every pair of distinct vertices is adjacent. A clique of size $n$ is referred to as a complete graph $K_n$.
A \emph{coclique} of $\Gamma$ is a subset of $V(\Gamma)$ such that no two vertices are adjacent. 
A \emph{complete bipartite graph} $K_{m,n}$ is a graph whose vertex set can be partitioned into two cocliques, say an $m$-set $M$ and an $n$-set $N$, where each vertex in $M$ is adjacent to all vertices in $N$.
A \emph{complete multipartite graph} $K_{t\times n}$ is a graph whose vertex set can be partitioned into cocliques $\{M_i\}^t_{i=1}$ of size $n$, where each vertex in $M_i$ is adjacent to all vertices in $M_j$ $(1\leq j\ne i \leq t)$.
We note that $K_{2\times m}$ is the same as $K_{m,m}$.

\smallskip
Next, we recall the concepts of a local graph and a $\mu$-graph.
For a vertex $x \in V(\Gamma)$, let $\Delta(x)$ denote the subgraph of $\Gamma$ induced on $\Gamma(x)$. 
We call $\Delta(x)$ the \emph{local graph} of $\Gamma$ at $x$. 
Let $\mathcal{P}$ be a property of a graph or a family of graphs. We say $\Gamma$ is \emph{locally} $\mathcal{P}$ whenever every local graph of $\Gamma$ has the property $\mathcal{P}$ or belongs to the family $\mathcal{P}$.
For example, we say $\Gamma$ is locally complete multipartite or locally strongly regular.
Suppose that $\Gamma$ is amply regular with parameters $(\nu,\kappa,\lambda,\mu)$.
For two vertices $x$, $y$ with $\partial(x,y)=2$, the subgraph of $\Gamma$ induced on $\Gamma(x) \cap \Gamma(y)$ is called a \emph{$\mu$-graph} of $\Gamma$.
If $\Gamma$ is distance-regular, a $\mu$-graph is often called a \emph{$c_2$-graph} of $\Gamma$.

\begin{lemma}[{\cite[Proposition 1.3.2]{bcn89}}]\label{prop:Hoffman bound}
Let $\Gamma$ be a regular graph with $v$ vertices, valency $k$, and smallest eigenvalue $-m$.
\begin{enumerate}[\normalfont(i)]
	\item If $C$ is a coclique of $\Gamma$, then $|C|\leq v(1 + k/m)^{-1}$, with equality if and only if every vertex outside $C$ has exactly $m$ neighbors in $C$.
	\item If $\Gamma$ is strongly regular and $C$ is a clique of $\Ga$, then 
	\begin{equation}\label{Hoff bound}
		|C|\leq 1+{k}/{m},
	\end{equation} 
	with equality if and only if every vertex outside $C$ has exactly $\mu / m$ neighbors in $C$, where $\mu$ is the number of common neighbors of any two nonadjacent vertices.
\end{enumerate}
\end{lemma}
\noindent
The upper bound for the size of a clique in \eqref{Hoff bound} is called the \emph{Hoffman bound} (or \emph{Delsarte bound}). 
If a clique $C$ in a distance-regular graph attains the Hoffman bound, we call $C$ a \emph{Delsarte} clique.

\begin{lemma}\label{lem:C}
Let $\Gamma$ be an amply regular graph with parameters $(\nu, k, a_1, c_2)$.
Assume that $\Gamma$ is locally strongly regular with parameters $(k, a_1,\lambda, \mu)$. 
For a vertex $x$ of $\Gamma$, let $\Delta(x)$ be the local graph of $\Gamma$ at $x$ with smallest eigenvalue $-m$. 
If $C$ is a Delsarte clique of $\Delta(x)$, then a vertex at distance two from $x$ either has $1+\mu/m$ neighbors in $C$ or no neighbors in $C$.
\end{lemma}
\begin{proof}
Let $z$ be a vertex of $\Gamma$ at distance two from $x$.
Suppose that the Delsarte clique $C$ has a neighbor of $z$.
We will show that the number of neighbors of $z$ in $C$ is $1+\mu/m$.
Select a vertex $y\in C$ that is adjacent to $z$.
Consider the local graph $\Delta(y)$ in $\Gamma$, and note that $\Delta(y)$ is strongly regular with smallest eigenvalue $-m$. 
Now, consider the vertex subset $C' = C \cup \{x\} \setminus \{y\}$ in $\Gamma$. 
Obviously, $C'$ forms a clique in $\Delta(y)$ of the same size as $C$. 
Hence, $C'$ is a Delsarte clique of $\Delta(y)$. 
Since $\Delta(y)$ is strongly regular and $z\in \Delta(y)$ is not an element of $C'$, Lemma \ref{prop:Hoffman bound}(ii) implies that $z$ has $\mu/m$ neighbors in $C'$.
Therefore, $z$ has precisely $1+\mu/m$ neighbors in $C$.
\end{proof}

We recall the $Q$-polynomial property.
Let $\Gamma$ be distance-regular with diameter $D \geq 3$.
We abbreviate the vertex set as $X = V(\Gamma)$.
We denote $\MxR$ as the $\mathbb{R}$-algebra consisting of real matrices, where both rows and columns are indexed by $X$.
For each integer $0\leq i \leq D$, define the matrix $A_i \in \MxR$ with $(x,y)$-entry $1$ if $\partial(x,y)=i$ and $0$ otherwise.
Observe that 
$$
	A_iA_j = \sum^D_{h=0} p^h_{i,j} A_h \qquad (0\leq i,j \leq D).
$$
It is known that the matrices $\{A_i\}^D_{i=0}$ form a basis for a commutative subalgebra $M$ of $\MxR$.
We call $M$ the \emph{Bose-Mesner algebra of $\Gamma$}.
The algebra $M$ has a second basis $\{E_i\}^D_{i=0}$ such that $E_iE_j=\delta_{i,j}E_i$ $(0\leq i,j \leq D)$, where the matrices $E_i$ $(0\leq i \leq D)$ are called the {primitive idempotents} of $\Gamma$.
We note that $M$ is closed under the entrywise multiplication $\circ$ since $A_i\circ A_j =\delta_{i,j}A_i$.
Thus, there exist real numbers $q^h_{i,j}$ such that
$$
	E_i\circ E_j = |X|^{-1}\sum^{D}_{h=0} q^h_{i,j} E_h \qquad (0 \leq i,j \leq D).
$$
An ordering $\{E_i\}^D_{i=0}$ is called \emph{$Q$-polynomial} whenever for all distinct $h,j$ $(0\leq h,j\leq D)$ we have $q^h_{1,j} =0$ if and only if $|h-j| \ne 1$.
We say $\Gamma$ is \emph{$Q$-polynomial} whenever there is a $Q$-polynomial ordering of the primitive idempotents; cf. \cite[p.~235]{bcn89}.
Suppose $\Gamma$ is a tight distance-regular graph.
In \cite{Pasc2001}, several characterizations of $\Gamma$ with the $Q$-polynomial property were introduced.
In \cite[Section 13(vi)]{AJJKPT}, the authors provided many examples of $\Gamma$, both with and without the $Q$-polynomial property. 
Here, we recall one example of $\Gamma$ that does not have the $Q$-polynomial property, which will be used later in this paper.
\begin{example}[{\cite[Section 13.2.D]{bcn89}}]\label{ex:3.O7(3)}
The graph $3.O_7(3)$ is distance-transitive with $1134$ vertices and has intersection array $\{117,80,24,1;1,12,80,117\}$.
The graph $3.O_7(3)$ is tight but not $Q$-polynomial.
Each local graph of $3.O_7(3)$ is strongly regular with parameters $(117, 36, 15, 9)$, and has nontrivial eigenvalues $r = 9$, $s =-3$. 
\end{example}

We finish this section with one comment.
Let $\Gamma$ be a graph with valency $k$ and diameter $D$.
It is well-known that the number of vertices is bounded in terms of $k$ and $D$:
\begin{equation}\label{MooreBound}
	|V(\Gamma)| 
	\leq 1 + k + k(k-1) + \cdots + k(k-1)^{D-1}
	= \begin{cases}
	1+\dfrac{k((k-1)^D-1)}{k-2} & k>2; \\
	2D+1 & k=2.
	\end{cases}
\end{equation}
The right-hand side of \eqref{MooreBound} is called the \emph{Moore bound}.
We call $\Gamma$ a \emph{Moore graph} if the equality in \eqref{MooreBound} holds.
For more detailed information about Moore graphs, see \cite{MillerSiran2013}.

%%%%%%%%%%%%%%%%%%%%%%%%%%%%%%%%%%%%%%%%%
%%%%%%%%%%%%%%%%%%%%%%%%%%%%%%%%%%%%%%%%%
%%%%%%%%%%%%%%%%%%%%%%%%%%%%%%%%%%%%%%%%%
% new section
\section{The block graph of an orthogonal array}\label{Sec:LS(m,n)}
In this section, we discuss the block graph of an orthogonal array and its properties.
We then analyze the structure of the $\mu$-graphs of an amply regular graph that is locally the block graph of an orthogonal array.
An \emph{orthogonal array}, denoted as $\operatorname{OA}(m,n)$, is a structured $m\times n^2$ array with entries chosen from the set $\{1,\ldots, n\}$. 
It possesses the property that the columns of every $2\times n^2$ subarray contain all possible $n^2$ pairs exactly once. 
In other words, for each pair of rows, every pair of elements from the set $\{1,\ldots, n\}$ appears precisely once in a column.
The \emph{block graph of an orthogonal array} is a graph whose vertices are the columns of $\operatorname{OA}(m,n)$, where two columns are adjacent if and only if there exists a row where they share the same entry.
We note that the block graph of $\operatorname{OA}(m,n)$ is the same concept as the Latin square graph $L_m(n)$; see \cite[Section 8.4]{SRG}.

\begin{lemma}[cf. {\cite[Theorem 5.5.1]{EKR}}]\label{prop: XOA(m,n)}
If $\operatorname{OA}(m,n)$ is an orthogonal array with $n\geq m$, then its block graph is a strongly regular graph with parameters
\begin{equation}\label{eq:propXOA(m,n)}
	\left(n^2, ~m(n-1), ~(m-1)(m-2)+n-2, ~m(m-1)\right).
\end{equation}
Moreover, the spectrum of the block graph of $\operatorname{OA}(m,n)$ is 
\begin{equation*}
	\begin{pmatrix}
	m(n-1) & n-m & -m \\
	1 & m(n-1) & (n-1)(n+1-m)
	\end{pmatrix}.
\end{equation*}
\end{lemma}

\noindent
Using Lemma \ref{prop:Hoffman bound}(ii) and Lemma \ref{prop: XOA(m,n)}, we find that the maximum clique size in the block graph of $\operatorname{OA}(m,n)$ is $n$.
Constructing a Delsarte clique in the block graph of $\operatorname{OA}(m,n)$ is straightforward: for each $i\in \{1,\ldots,n\}$, consider the set $S_{r,i}$, which consists of the columns of $\operatorname{OA}(m, n)$ containing the entry $i$ in row $r$.
Note that these sets naturally form cliques. 
Furthermore, as each element in $\{1,\ldots,n\}$ appears exactly $n$ times in each row, the size of each clique $S_{r,i}$ is $n$ for all $i$ and $r$. 
These cliques are referred to as the \emph{canonical cliques} of the block graph of $\operatorname{OA}(m,n)$.

\begin{lemma}\label{mainlem1}
Let $\Gamma$ be an amply regular graph with parameters $(v, k, a_1, c_2)$ and locally the block graph of an orthogonal array $\operatorname{OA}(m, n)$.
If $c_2 = m^2$, then every $c_2$-graph of $\Gamma$ is the block graph of an orthogonal array $\operatorname{OA}(m, m)$, and therefore, is complete $m$-partite.
\end{lemma}

\begin{proof} 
Observe that for each row $r$ ($1 \leq r \leq m$) in $\operatorname{OA}(m,n)$, the set $S_{r,i}$ ($1 \leq i \leq n$) forms a canonical clique of size $n$.
Fix a vertex $x$ of $\Gamma$, and let $\Delta$ denote the local graph of $\Gamma$ at $x$.
By construction of $\operatorname{OA}(m,n)$, $\Delta$ consists of $n$ (disjoint) canonical cliques 
\begin{equation*}
		S_{r,1}, S_{r,2}, \ldots, S_{r,n} \qquad \qquad (1\leq r \leq m). 
\end{equation*}
Note that every vertex of $\Delta$ belongs to exactly $m$ canonical cliques.
Fix a row $r=1$ and observe that each $S_{1,i}$ is a canonical clique in $\Delta$.
Select a vertex $z$ of $\Gamma$ at distance two from the vertex $x$.  
Let $\mathrm{M}=\mathrm{M}(x,z)$ denote the $c_2$-graph of $\Gamma$ induced by the vertices $x$ and $z$.
Since $c_2 = m^2$, $\mathrm{M}$ consists of $m^2$ columns obtained from the orthogonal array $\operatorname{OA}(m,n)$. 
Let $\mathcal{O}$ be the $m \times m^2$ array consisting of the vertices of $\mathrm{M}$.
We claim that $\mathcal{O}$ has the structure of an orthogonal array $\operatorname{OA}(m,m)$, which implies that $\mathrm{M}$ is a block graph of $\operatorname{OA}(m,m)$. 
To prove this claim, we will show that in each row of $\mathcal{O}$, precisely $m$ distinct symbols occur, each exactly $m$ times.% (adjusting some entries in $\{m+1, \ldots, n\}$ if necessary by replacing them with entries from $\{1, \ldots, m\}$). 
In other words, it is equivalent to proving that $\mathrm{M}$ consists of $m$ disjoint canonical cliques, with each vertex of $\mathrm{M}$ being incident to precisely $m$ canonical cliques. 

\smallskip
\noindent
For $1\leq i \leq n$, define $C_i:=S_{1,i} \cap \Gamma(z)$. 
Applying Lemma \ref{lem:C}, we find that for each $i$, the size of $C_i$ is either $m$ or $0$. 
Observe that $C_i$ forms a canonical clique of $\mathrm{M}$ if its size is $m$. 
Therefore, $\{C_i \mid 1\leq i\leq n, C_i \neq \varnothing\}$ is a partition of the vertex set of $\mathrm{M}$ into $m$ canonical cliques of size $m$. 
Note that, without loss of generality, we may permute the entries of $\operatorname{OA}(m,n)$ so that $C_i = \varnothing$ for all $i>m$, and thus $\mathcal{O}$ consists of the entries $\{1,2, \ldots, m\}$ and each vertex in $\mathrm{M}$ is incident to $m$ canonical cliques.
Therefore, we conclude that $\mathrm{M}$ is the block graph of $\operatorname{OA}(m,m)$.
%{\color{red}
%Choose a vertex $y_1$ in $\mathrm{M}$. There exists a canonical clique in $\Delta$ containing $y_1$, denoted by $S_{1,1}$. By Lemma \ref{prop: XOA(m,n)}, $\Delta$ is a strongly regular graph with $\mu = m(m-1)$. Applying Lemma \ref{lem:C}, we find that $S_{1,1}$ contains $m$ neighbors of $z$. Let ${C}_1 := S_{1,1} \cap \Gamma(z)$, the set of $m$ neighbors of $z$ in $S_{1,1}$. Observe that ${C}_1$ forms a canonical clique in $\mathrm{M}$ of size $m$.
%Next, choose a vertex $y_2$ from $\mathrm{M} \setminus {C}_1$. There exists a canonical clique in $\Delta$ containing $y_2$, denoted by $S_{1,2}$. By Lemma \ref{lem:C}, $S_{1,2}$ also contains $m$ neighbors of $z$, which are denoted by ${C}_2$. Similarly, ${C}_2$ forms a canonical clique in $\mathrm{M}$ of size $m$. Repeat this process. Since $\mathrm{M}$ has $m^2$ vertices, we obtain exactly $m$ disjoint cliques
%$$
%	{C}_1, {C}_2, \ldots, {C}_m,
%$$
%where each ${C}_i$ has a size of $m$. 
%Consequently, $\mathrm{M}$ is composed of $m$ cliques $\{{C}_i\}^m_{i=1}$.
%Note that this sequence of cliques $\{{C}_i\}^m_{i=1}$ is determined by choosing vertices $y_1, y_2, \ldots, y_m$.
%If a vertex of $\mathrm{M}$ has an entry from $\{m+1, m+2, \dots, n\}$, we replace it with an appropriate entry from $\{1, \ldots, m\}$ so that the graph structure of $\mathrm{M}$ is well-preserved.
%Through this construction, each vertex in $\mathrm{M}$ is incident to $m$ canonical cliques. }
\end{proof}

%%%%%%%%%%%%%%%%%%%%%%%%%%%%%%%%%%%%%%%%%
%%%%%%%%%%%%%%%%%%%%%%%%%%%%%%%%%%%%%%%%%
%%%%%%%%%%%%%%%%%%%%%%%%%%%%%%%%%%%%%%%%%
% new section
\section{The block graph of a Steiner system}\label{Sec:Steiner S(m,n)}
In this section, we discuss the block graph of a Steiner system and its properties.
We then analyze the structure of $\mu$-graph of an amply regular graph that is locally the block graph of a Steiner system.
A \emph{Steiner system $S(2,m,n)$} is a \emph{$2$-$(n,m,1)$ design}, that is, a collection of $m$-sets taken from a set of size $n$, satisfying the property that every pair of elements from the $n$-set is contained in exactly one $m$-set. 
In this context, the elements of the $n$-set are referred to as \emph{points}, and the $m$-sets are referred to as \emph{blocks} of the system.
A straightforward counting argument reveals that the number of blocks in a Steiner system $S(2,m,n)$ is given by ${n(n-1)}/{m(m-1)}$, and each point occurs in exactly ${(n-1)}/{(m-1)}$ blocks.
A Steiner system $S(2,m,n)$ is said to be \emph{symmetric} if the number of points is equal to the number of blocks; otherwise, it is regarded as \emph{non-symmetric}.
The \emph{block graph of a Steiner system $S(2,m,n)$} is defined as the graph whose vertices correspond to the blocks of the system. 
Two blocks are adjacent in this graph if and only if they intersect at exactly one point.

\begin{lemma}[cf. {\cite[Theorem 5.3.1]{EKR}}]\label{prop:blockgraphdesign}
The block graph of a non-symmetric Steiner system $S(2,m,n)$ is a strongly regular graph with parameters
\begin{equation}\label{eq:S(2,m,n)parameters}
	\left( \frac{n(n-1)}{m(m-1)}, ~\frac{m(n-m)}{m-1},~ (m-1)^2+\frac{n-1}{m-1}-2,~ m^2 \right).
\end{equation}
Moreover, the spectrum of this graph is 
\begin{equation}\label{eq:spec S(2,m,n)}
	\begin{pmatrix}
	\frac{m(n-m)}{m-1} & \frac{n-m^2}{m-1} & -m \\
	1 & n-1 & \frac{n(n-1)}{m(m-1)}-n
	\end{pmatrix}.
\end{equation}
\end{lemma}

The block graph of a Steiner system $S(2,m,mn + m -n)$ with $n\geq m+1$ is called a \emph{Steiner graph $S_m(n)$}.
By Lemma \ref{prop:blockgraphdesign}, the graph $S_m(n)$ is strongly regular with parameters
\begin{equation}\label{eq:Sm(n)parameters}
		\left(\frac{(m+n(m-1))(n+1)}{m}, \ mn, \ m^2-2m+n, \ m^2\right).
\end{equation}

\noindent
Using \eqref{Hoff bound} and \eqref{eq:spec S(2,m,n)}, we can determine that the size of a maximum clique in the block graph of a Steiner system $S(2,m,n)$ is $({n-1})/({m-1})$. 
Constructing a Delsarte clique in the block graph of $S(2,m,n)$ is straightforward: for each $i \in \{1, \ldots , n\}$, we define $S_i$ as the set of all blocks in the design that contain the point $i$. 
These cliques $S_i$ are referred to as the \emph{canonical cliques} of the block graph.

\begin{lemma}\label{mainlem2}
Let $\Ga$ be an amply regular graph with parameters $(v,k,a_1,c_2)$ and locally the block grpah of a Steiner system $S(2,m,n)$.
If $c_2=m(m+1)$, then every $c_2$-graph of $\Ga$ is the block graph of a Steiner system $S(2,m,m^2)$, and therefore, is complete $(m+1)$-partite.
\end{lemma}

\begin{proof}
For a vertex $x$ of $\Gamma$, let $\Delta$ denote the local graph of $\Gamma$ at $x$, that is, the block graph of a Steiner system $S(2,m,n)$.
We denote its corresponding Steiner system by $(\mathcal{P},\mathcal{B})$, where $\mathcal{P}$ denotes the set of points and $\mathcal{B}$ denotes the set of blocks. 
Observe that $\mathcal{B}$ is the vertex set of the local graph $\Delta$, and furthermore, $|\mathcal{P}| = n$ and $|\mathcal{B}| = n(n-1)/(m(m-1))$.
Select a vertex $y$ of $\Gamma$ at distance two from the vertex $x$.
Let $\mathrm{M}(x,y)$ denote the $c_2$-graph of $\Gamma$ induced by the vertices $x$ and $y$.
Let $\mathcal{B}'$ denote the vertex set of $\mathrm{M}(x,y)$. 
Observe that $\mathcal{B}'$ is a subset of $\mathcal{B}$ with cardinality $m(m+1)$ since $c_2=m(m+1)$.
We define the subset $\mathcal{P}'$ of $\mathcal{P}$ by
$$
	\mathcal{P}'=\left\{p\in \mathcal{P} \ \middle| \ p \in \bigcup_{B\in \mathcal{B}'} B \right\}.
$$
We claim that $|\mathcal{P}'|=m^2$. 
To prove this claim, let us consider a vertex $B$ in $\mathrm{M}(x,y)$. 
Since $B$ is a block in $\mathcal{B}'$, we can write it as $B = \{ p_1, p_2, \ldots, p_m \}$, where $p_i \in \mathcal{P}'$ ($1\leq i \leq m$).
Now, for the point $p_1$ we consider the canonical clique $S_{p_1}$ of $\Delta$. 
By Lemma \ref{prop:blockgraphdesign} and \eqref{eq:S(2,m,n)parameters}, $\Delta$ is strongly regular with $\mu=m^2$. 
Applying Lemma \ref{lem:C}, we find that there are exactly $m+1$ neighbors of $y$ in $S_{p_1}$, denoted as $B=B_0, B_1, \ldots, B_m$.
Observe that each $B_i$ contains $m-1$ points, excluding the common point $p_1$. 
It implies that the total number of points in $\bigcup^m_{i=0} B_i$ is $m^2$. 
Since each $B_i$ belongs to $\mathcal{B}'$, all $m^2$ points are elements of $\mathcal{P}'$. 
Therefore, we have $|\mathcal{P}'| \geq m^2$.

\smallskip
\noindent
Suppose that $|\mathcal{P}'| > m^2$. 
Recall the vertices $B = \{p_1, p_2, \ldots, p_m\}, B_1, \ldots, B_m$.
For $1\leq i\leq m$, let $S_{p_i}$ denote the canonical clique of $\Delta$ corresponding to the point $p_i$.
By construction, the canonical cliques containing the vertex $B$ are precisely $S_{p_1}, S_{p_2},\ldots, S_{p_m}$, and each $S_{p_i}$ has precisely $m$ neighbors of $y$ besides $B$.
Therefore, we obtain $m^2+1$ vertices of $\mathrm{M}(x,y)$.
Now, choose a point $q \in \mathcal{P}'$ such that $q \notin B_i$ for all $0\leq i \leq m$.
Such a point can be chosen because $|\bigcup^m_{i=0}B_i|=m^2$ and by our assumption $|\mathcal{P}'|> m^2$.
Note that none of the points of $p_1, p_2, \ldots, p_m$ equals $q$.
Consider the corresponding canonical clique $S_q$ of $\Delta$.
It follows that none of $S_{p_1}, S_{p_2}, \ldots, S_{p_m}$ equals $S_q$.
By Lemma \ref{lem:C}, $S_q$ has $m+1$ neighbors of $y$, denoted as $\check{B}_0, \check{B}_1, \ldots, \check{B}_m$. 
These blocks $\{\check{B}_i\}^m_{i=0}$ belong to $\mathcal{B}'$, and each block $\check{B}_i$ contains the point $q$, so we obtain $m+1$ new vertices in $\mathrm{M}(x,y)$.
This implies that the number of vertices of $\mathrm{M}(x,y)$ is at least $(m^2+1) + (m+1) = m^2+m+2$. However, this contradicts the fact that $|\mathcal{B}'|=c_2=m^2+m$. 
Hence, we conclude that $|\mathcal{P}'|=m^2$, as claimed.

\smallskip
\noindent
Next, we consider the pair $(\mathcal{P}', \mathcal{B}')$. 
We will show that this pair forms a $2$-$(m^2,m,1)$ design, that is, each pair of points in $\mathcal{P}'$ is contained in exactly one block of $\mathcal{B}'$.
For each pair of distinct points $p$ and $q$ in $\mathcal{P}'$, let $B_{p,q}$ denote the (unique) block in $\mathcal{B}$ that contains both $p$ and $q$. 
We define $\mathcal{B}''$ as the collection of blocks in $\mathcal{B}$ that contain pairs of points from $\mathcal{P}'$, i.e., $\mathcal{B}''=\{ B_{p,q} \in \mathcal{B} \mid p,q \in \mathcal{P}'\}$. 
We assert that $\mathcal{B}'=\mathcal{B}''$.
First, it is clear that $\mathcal{B}'$ is a subset of $\mathcal{B}''$. 
Next, we determine the cardinality of $\mathcal{B}''$. 
To do this, consider the set $\left\{( \{p,q\}, B) \, \middle| \, B \in \mathcal{B}'', \{p,q\} \in {B\choose2}\right\}$.
Through double-counting the pairs $(\{p,q\},B)$, we find
$$
	{m \choose 2}|\mathcal{B}''| \leq {|\mathcal{P}'| \choose 2}.
$$
Simplifying this inequality, we obtain $|\mathcal{B}''| \leq m(m+1)$.
On the other hand, since $\mathcal{B'}\subseteq \mathcal{B''}$ and $|\mathcal{B}'|=m(m+1)$, it follows that $|\mathcal{B}''|=m(m+1)$.
Therefore, we have $\mathcal{B}'=\mathcal{B}''$, as asserted.
Consequetly, the pair $(\mathcal{P}', \mathcal{B}')$ possesses the structure of a $2$-$(m^2,m,1)$ design.
The result follows.
\end{proof}

%%%%%%%%%%%%%%%%%%%%%%%%%%%%%%%%%%%%%%%%%
%%%%%%%%%%%%%%%%%%%%%%%%%%%%%%%%%%%%%%%%%
%%%%%%%%%%%%%%%%%%%%%%%%%%%%%%%%%%%%%%%%%
% new section
\section{Proof of Theorem \ref{main thm1}}\label{Sec:main thm1}

In this section, we prove Theorem \ref{main thm1}. 
To do this, we first recall and present some lemmas required for the proof without providing their proofs.

\begin{lemma}[cf. {\cite[Lemma 4]{JMT}}]\label{lem4JMT}
For integers $t, n \geq 2$ let $\Gamma$ be a connected graph of diameter at least $2$, in which every $\mu$-graph is isomorphic to $K_{t \times n}$. 
Then $\Gamma$ is regular.
Moreover, for an arbitrary vertex $x$ of $\Gamma$, the local graph $\Delta$ of $\Gamma$ at $x$ satisfies the following properties:
\begin{enumerate}[\normalfont(i)]
\item $\Delta$ is regular;
\item $\Delta$ has diameter $2$ and every $\mu$-graph of $\Delta$ is isomorphic to $K_{(t-1)\times n}$;
\item $\Delta$ is strongly regular if $t\geq 3$;
\item if the intersection number $\gamma(\Gamma)$ exists, then $\gamma(\Gamma)>0$ and the intersection number $\gamma(\Delta)$ exists with $\gamma(\Delta)=\gamma(\Gamma)-1$.
\end{enumerate}
\end{lemma}

\begin{lemma}[cf. {\cite[Theorem 8]{JMT}}]\label{lem8JMT}
For integers $t,n\geq 2$ let $\Gamma$ be a connected graph in which every $\mu$-graph is isomorphic to $K_{t\times n}$.
If the intersection number $\gamma(\Gamma)$ exists with $\gamma(\Gamma)\geq 2$, then $\gamma(\Gamma)=t$.
\end{lemma}

\begin{lemma}[cf. {\cite[Theorem 11]{JMT}}]\label{thm11JMT}
For an integer $n\geq 3$ let $\Gamma$ be a connected graph in which every $\mu$-graph is isomorphic to $K_{n,n}$. If the intersection number $\gamma(\Gamma)$ exists and $\gamma(\Gamma)=2$, then $\Gamma$ is locally $\operatorname{GQ}(\lambda/n,n-1)$.
In particular, $\Gamma$ has diameter $2$ if and only if $\Gamma$ is locally $\operatorname{GQ}(n-1,n-1)$.
\end{lemma}

\begin{lemma}[cf. {\cite[Theorem 12]{JMT}}]\label{thm12JMT}
For integers $t\geq 1$ and $n\geq 3$ let $\Gamma$ be a connected graph in which every $\mu$-graph is isomorphic to $K_{t\times n}$. 
If the intersection number $\gamma(\Gamma)$ exists, then $t\leq 4$.
Moreover, equality holds only if $\Gamma$ is the unique distance-regular graph $3.O_7(3)$, which is locally locally locally $\operatorname{GQ}(2,2)$.
\end{lemma}

%{\color{red}
%\begin{lemma}[cf. {\cite[Theorem 8.6.4]{SRG}}]\label{SRG Thm8.6.4}
%Let $\Gamma$ be a strongly regular graph with integral smallest eigenvalue $-m$, where $m\geq 2$.
%Then $\Gamma$ belongs to one of the following:
%\begin{enumerate}[\normalfont(i)]
%	\item Complete multipartite graphs with classes of size $m$.
%	\item Block graphs of orthogonal arrays $\operatorname{OA}(m,n)$.
%	\item Block graphs of Steiner systems $S(2,m,mn+m-n)$.
%	\item Finitely many further graphs.
%\end{enumerate}
%\end{lemma}
%}

Now we are ready to prove Theorem \ref{main thm1}.
\begin{proof}[Proof of Theorem \ref{main thm1}]
Let $\Delta$ denote the local graph of $\Gamma$ at a vertex $x\in V(\Gamma)$.
Since $\Delta$ is strongly regular, we denote its parameters as $(k,a_1,\lambda, \mu)$ and its eigenvalues as $a_1 > r > -m$, where $a_1$ is the intersection number of $\Gamma$. 
For notational convenience, we let $n=r+m$. 
Now, we consider each case: (i) $\Delta$ is the block graph of an orthogonal array, and (ii) $\Delta$ is the block graph of a Steiner system.

\smallskip
\noindent
\textbf{Case (i):} 
Suppose $\Delta$ is the block graph of an orthogonal array with $k > m^2$.
Assume that $c_2=m^2$; we will derive a contradiction from this assumption.
To this end, we consider the $c_2$-graphs of $\Gamma$.
By Lemma \ref{mainlem1}, every $c_2$-graph of $\Gamma$ is the block graph of $\operatorname{OA}(m,m)$, which is isomorphic to $K_{m\times m}$, where $m\geq 3$.

\smallskip
\noindent
We claim that $m=3$.
To show this, we consider the (triple) intersection number $\gamma(\Gamma)$.
We assert that $\gamma(\Gamma) \geq 2$.
Suppose that $\gamma(\Gamma)=1$.
Choose a vertex $z$ at distance two from $x$, and then choose a vertex $y$ that is adjacent to both $x$ and $z$.
Next, choose a Delsarte clique $C$ of $\Delta$ that contains $y$.
Consider the subset $N_z:=C \cap \Gamma(z)$ of $C$. 
Note that $N_z$ is not empty since $y\in N_z$.
By Lemma $\ref{lem:C}$, and since $\mu=m(m-1)$ by \eqref{eq:propXOA(m,n)}, we have $|N_z|=1+\mu/m=m$.
Since $n>m$, one can choose a vertex $y' \in C \setminus N_z$.
Considering the triple of vertices $(x,y',z)$ and using the assumption $\gamma(\Gamma)=1$, it follows that $N_z=\{y\}$. 
Thus, $|N_z|=m=1$, which contradicts $m\geq 3$.
Therefore, we have $\gamma(\Gamma)\geq 2$, as asserted.
Since the $c_2$-graph of $\Gamma$ is isomorphic to $K_{m\times m}$ and the intersection number $\gamma(\Gamma)$ exists with $\gamma(\Gamma)\geq2$, by applying Lemma \ref{lem8JMT} to $\Gamma$ we obtain $\gamma(\Gamma)=m$.
In addition, applying Lemma \ref{thm12JMT} to $\Gamma$ and considering the given condition $m\geq 3$, we have $3\leq m \leq 4$.
If $m=4$, by Lemma \ref{thm12JMT}, $\Gamma$ must be the distance-regular graph $3.O_7(3)$.
In this case, referring to Example \ref{ex:3.O7(3)}, $\Delta$ has the smallest eigenvalue $-3$, namely $m=3$, contradicting the given $m=4$. 
Therefore, we rule out the case $m=4$.
Consequently, we have $m=3$, as claimed.

\smallskip
\noindent
From the claim, it follows that the $c_2$-graph of $\Gamma$ is isomorphic to $K_{3\times 3}$.
With this comment, we apply Lemma \ref{lem4JMT} to $\Gamma$, obtainining that every $\mu$-graph of $\Delta$ is isomorphic to $K_{2\times 3}$, and the intersection number $\gamma(\Delta)$ exists with $\gamma(\Delta)=\gamma(\Gamma)-1=3-1=2$.
Subsequently, by applying Lemma \ref{thm11JMT} to $\Delta$, we conclude that $\Delta$ is locally $\operatorname{GQ}(2,2)$.

\smallskip
\noindent
However, this is impossible for the following reasons. 
Choose a vertex $v$ in $\Delta$ and consider the local graph $\Delta(v)$ of $\Delta$ at $v$.
Then $\Delta(v)$ is $\operatorname{GQ}(2,2)$, a strongly regular graph with parameters $(15,6,1,3)$.
By \eqref{Hoff bound}, the maximal size of a clique of $\Delta(v)$ is $3$.
But we can find a clique of size $5$ within $\Delta(v)$ as follows.
Consider a Delsarte clique $C$ of $\Delta$ containing $v$.
Since $|\Delta(v)|=15$, it follows that $a_1=15$, which is the valency of $\Delta$.
Recall $m=3$, where $-m$ is the smallest eigenvalue of $\Delta$.
By \eqref{Hoff bound}, we have $|C|=1+a_1/m = 6$.
Since $C\setminus\{v\}$ is a clique in $\Delta(v)$, we find that $\Delta(v)$ contains a clique of size $5$.
This contradicts the requirement that the maximal size of a clique in $\Delta(v)$ is $3$.
Therefore, $\Delta$ cannot be locally $\operatorname{GQ}(2,2)$.
Consequently, we conclude $c_2 \neq m^2$.

\smallskip
\noindent
\textbf{Case (ii):} The proof is similar to Case (i).
Suppose $\Delta$ is the block graph of a Steiner system with $k>m(m+1)$.
%{\color{red} \sout{Let $\Delta$ be the block graph of a Steiner system. By Lemma \ref{SRG Thm8.6.4} and since $k>m(m+1)$, the graph $\Delta$ is the Steiner graph $S_m(n)$ with $n>m$.} } %%
Assume that $c_2=m(m+1)$.
By Lemma \ref{mainlem2}, every $c_2$-graph of $\Gamma$ is the block graph of a Steiner system $S(2,m,m^2)$, which is isomorphic to $K_{m\times (m+1)}$.
We determine the intersection number $\gamma(\Gamma)$.
Using the same argument as in the proof of Case (i), we find that $\gamma(\Gamma)=m=3$.
Therefore, every $c_2$-graph of $\Gamma$ is isomorphic to $K_{3\times 4}$.
By Lemma \ref{lem4JMT}, every $\mu$-graph of $\Delta$ is isomorphic to $K_{2\times 4}$ and  the intersection number $\gamma(\Delta)$ is $2$.
Therefore, by Lemma \ref{thm11JMT}, $\Delta$ is locally $\operatorname{GQ}(3,3)$.
However, this is impossible for the following reasons.
Choose a vertex $v$ in $\Delta$.
Then, the local graph $\Delta(v)$ of $\Delta$ at $v$ is $\operatorname{GQ}(3,3)$, a strongly regular graph with parameters $(40,12,2,4)$.
Therefore, the valency of $\Delta$ is $40$.
By \eqref{eq:S(2,m,n)parameters} and since $m=3$, the valency of $\Delta$ is $3(n-3)/2$.
From these comments, we have $3(n-3)/2=40$, which implies $n=89/3$. 
%{\color{red} \sout{By \eqref{eq:Sm(n)parameters} and since $m=3$, the valency of $\Delta$ is $3n$. From these comments, we have $3n=40$, which implies $n=40/3$.} }
This contradicts the fact that $n$ is an integer.
Therefore, $\Delta$ cannot be locally $\operatorname{GQ}(3,3)$.
Consequently, we conclude $c_2\neq m(m+1)$.
The proof is now complete.
\end{proof}

\begin{remark}
In Theorem \ref{main thm1}, we assumed that $\Gamma$ is locally strongly regular with smallest eigenvalue $-m$, where $m \geq 3$.
In the proof of the theorem, assuming $c_2=m^2$ (resp. $c_2=m(m+1)$), 
we obtained that each $c_2$-graph of $\Gamma$ is the block graph of the orthogonal array $\operatorname{OA}(m,m)$ (resp. the Steiner system $S(2,m,m^2)$) from Lemma \ref{mainlem1} (resp. Lemma \ref{mainlem2}), and derived a contradiction from its structure.
It is worth noting that the existence of an orthogonal array $\operatorname{OA}(m,m)$ is equivalent to the existence of a projective plane of order $m$.
Similarly, the existence of a Steiner system $S(2,m,m^2)$ is equivalent to the existence of a projective plane of order $m$.
Thus, if $m$ is a number for which no projective plane of order $m$ exists, then the $c_2$-graph of $\Gamma$ does not exist, and hence we do not need the assumption that the intersection number $\gamma(\Gamma)$ exists.
\end{remark}

Next, we apply Theorem \ref{main thm1} to tight distance-regular graphs, resulting in the following.

\begin{corollary}\label{cor:sub-main}
Let $\Gamma$ be a tight distance-regular graph with diameter $D\geq 3$, intersection numbers $b_1, c_2$, and eigenvalues $k>\theta_1>\cdots>\theta_D$.
Define
\begin{equation*}
	b:=b_1/(1+\theta_1).
\end{equation*}
Assume  $b\geq 2$.  
Then the following {\rm(i)} and {\rm(ii)} hold.
\begin{enumerate}[\normalfont(i)]
	\item If $\Gamma$ is locally the block graph of an orthogonal array and $k>(b+1)^2$, then $c_2 \neq (b+1)^2$,
	\item If $\Gamma$ is locally the block graph of a Steiner system and $k>(b+1)(b+2)$, then $c_2 \neq (b+1)(b+2)$.
\end{enumerate}
\end{corollary}
\begin{proof}
Since $\Gamma$ is tight, it is locally connected strongly regular with smallest eigenvalue $-1-b$.
Moreover, the tight property implies that $\Gamma$ is 1-homogeneous, from which it follows that the intersection number $\gamma(\Gamma)$ exists.
With these comments, apply Theorem \ref{main thm1} to $\Gamma$.
The result follows.
\end{proof}

\begin{remark}
From Corollary \ref{cor:sub-main}, we conclude that a distance-regular graph $\Gamma$ with diameter at least $3$ and $b=b_1/(1+\theta_1) \geq 2$ cannot be tight if (i) $\Gamma$ is locally the block graph of an orthogonal array and $c_2 = (b+1)^2$, or (ii) $\Gamma$ is locally the block graph of a Steiner system and $c_2 = (b+1)(b+2)$.
%\begin{itemize}
%    \item[(i)] $\Gamma$ is locally the block graph of an orthogonal array and $c_2 = (b+1)^2$, or
%    \item[(ii)] $\Gamma$ is locally the block graph of a Steiner system and $c_2 = (b+1)(b+2)$.
%\end{itemize}
%As a consequence we have: If a DRG is tight, has classical parameters (D, b, \alpha, \beta) and b\geq 2, then the diameter (and valency) are bounded  by a function in b.
\end{remark}

We give a comment on the case when $\Gamma$ has diameter $D=3$ in Corollary \ref{cor:sub-main}.
Recall a \emph{Taylor graph}, that is, a distance-regular graph with intersection array $\{k,c_2,1;1,c_2;k\}$ with $c_2<k-1$.
We note that a nonbipartite distance-regular graph with diameter $3$ is tight if and only if it is a Taylor graph \cite[Theorem 3.2]{2002JKDiscMath}.
Let $\Gamma$ be a Taylor graph.
Then $\Gamma$ is locally strongly regular with parameters $(k, a_1, \lambda, \mu)$ and eigenvalues $a_1>r>s$.
Since $\Gamma$ is a Taylor graph, its local graphs satisfy
\begin{equation}\label{eq:Taylor para}
	a_1  = k-c_2-1, \qquad
	\lambda  = (3a_1-k-1)/2, \qquad
	\mu  = a_1/2,
\end{equation}
and
\begin{equation}\label{eq:Taylor k}
	k  = -(2r+1)(2s+1).
\end{equation}
%\begin{align}
%	& k  = -(2r+1)(2s+1), \label{eq:Taylor para(0)}\\  
%	& a_1  = k-c_2-1, \label{eq:Taylor para(1)}\\ 
%	& \lambda  = (3a_1-k-1)/2, \label{eq:Taylor para(2)}\\ 
%	& \mu  = a_1/2. \label{eq:Taylor para(3)}
%\end{align}
%We denote by $m=-s$ and $n=r-s$.
%Using \eqref{eq:Taylor para(0)}--\eqref{eq:Taylor para(3)} along with $\mu=a_1+rs$ from \eqref{eq:srg eigval}, we can express the parameters $(k, a_1, \lambda, \mu)$ of $\Delta$ in terms of $m$ and $n$:
%\begin{equation}\label{Taylor para:m,n}
%	\big( (2n-2m+1)(2m-1), \ 2m(n-m), \ (n-m)(m+1)-m, \ m(n-m) \big).
%\end{equation}
In Corollary \ref{cor:sub-main}, the graph $\Gamma$ with $D=3$ corresponds to a Taylor graph. 
In this case, referring to the above discussion, it can yield the following stronger result.

\begin{proposition}\label{prof:Taylor graph}
Let $\Gamma$ be a Taylor graph with intersection numbers $a_1$, $c_2$. 
Let $a_1>r>s$ denote the eigenvalues of a local graph of $\Gamma$.
%Let $\Delta$ be a local graph of $\Gamma$ with eigenvalues $a_1>r>s$.
Set $m=-s$ and $n=r-s$.
The following {\rm(i)}--{\rm(iii)} are equivalent:
\begin{enumerate}[\normalfont(i)]
	\item $\Gamma$ is locally strongly regular with the parameters of the block graph of $\operatorname{OA}(m,n)$,
	\item $n=2m-1$, and
	\item $c_2=2m(m-1)$.
\end{enumerate}
Furthermore, the following {\rm(iv)}--{\rm(vi)} are equivalent:
\begin{enumerate}
	\item[\rm(iv)] $\Gamma$ is locally strongly regular with the parameters of the Steiner graph $S_m(n)$,
	\item[\rm(v)]  $n=2m$, and 
	\item[\rm(vi)] $c_2=2(m+1)(m-1)$.
\end{enumerate}
\end{proposition}
\begin{proof}
Throughout this proof, let $\Delta$ denote a local graph of $\Gamma$ with parameters $(k, a_1, \lambda, \mu)$.
Using \eqref{eq:Taylor para}, \eqref{eq:Taylor k} along with $\mu=a_1+rs$ from \eqref{eq:srg eigval}, the parameters $(k, a_1, \lambda, \mu)$ are expressed in terms of $m$ and $n$:
\begin{equation}\label{Taylor para:m,n}
	\big( (2n-2m+1)(2m-1), \ 2m(n-m), \ (n-m)(m+1)-m, \ m(n-m) \big).
\end{equation}
First, we show that (i)--(iii) are equivalent.

\noindent
(i) $\Rightarrow$ (ii): Suppose $\Delta$ has parameters \eqref{eq:propXOA(m,n)} of the block  graph of $\operatorname{OA}(m,n)$.
Then we have $\mu=m(m-1)$.
Since $\Delta$ is the local graph of $\Gamma$, it also has the parameter $\mu=m(n-m)$ from \eqref{Taylor para:m,n}.
From these two formulas for $\mu$, it follows that $n=2m-1$.

\noindent
(ii) $\Rightarrow$ (iii): Suppose that $n=2m-1$.
Recall the parameters \eqref{Taylor para:m,n} of $\Delta$.
Substituting $n=2m-1$ into \eqref{Taylor para:m,n}, we obtain the parameters
\begin{equation}\label{eq(2):pf Taylor}
	\big( (2m-1)^2, \ 2m(m-1), \ m^2-m-1, \ m(m-1) \big).
\end{equation}
Observe that $c_2=k-a_1-1$ from the first equation in \eqref{eq:Taylor para}.
Evaluate $c_2$ using the parameters in \eqref{eq(2):pf Taylor} and simplify the result to get $c_2=2m(m-1)$.

\noindent
(iii) $\Rightarrow$ (i): Using $c_2=2m(m-1)$ and the parameters in \eqref{Taylor para:m,n}, express the equation $c_2=k-a_1-1$ in terms of $m$ and $n$ to obtain
\begin{equation}\label{eq(1):pf Taylor}
	2m(m-1) = (2n-2m+1)(2m-1) -  2m(n-m) - 1.
\end{equation}
Simplify \eqref{eq(1):pf Taylor} to get the equation $(m-1)(n-2m+1)=0$.
We note that $m\neq1$ since $-m$ is the smallest eigenvalue of $\Delta$.
Therefore, we have $n=2m-1$.
Using this equation, we find that the parameters in \eqref{eq:propXOA(m,n)} and \eqref{eq(2):pf Taylor} are equal. 
Therefore, $\Delta$ has the same parameters as the block graph of $\operatorname{OA}(m,n)$.

\smallskip
\noindent
Next, we show that (iv)--(vi) are equivalent. 

\noindent
(iv) $\Rightarrow$ (v): Suppose that $\Delta$ has parameters \eqref{eq:Sm(n)parameters} of the Steiner graph $S_m(n)$.
Then we have $\mu=m^2$.
Since $\Delta$ is the local graph of $\Gamma$, it also has the parameter $\mu=m(n-m)$ from \eqref{Taylor para:m,n}.
From these two formulas for $\mu$, it follows that $n=2m$.

\noindent
(v) $\Rightarrow$ (vi): Suppose that $n=2m$.
Substituting $n=2m$ into \eqref{Taylor para:m,n}, we obtain the parameters
\begin{equation}\label{eq(3):pf Taylor}
	\big( 4m^2-1, \ 2m^2, \ m^2, \ m^2 \big).
\end{equation}
Evaluate $c_2=k-a_1-1$ using the parameters in \eqref{eq(3):pf Taylor} and simplify the result to get $c_2=2(m+1)(m-1)$.

\noindent
(vi) $\Rightarrow$ (iv): Using $c_2=2(m+1)(m-1)$ and the parameters in \eqref{Taylor para:m,n}, express the equation $c_2=k-a_1-1$ in terms of $m$ and $n$ to obtain
\begin{equation}\label{eq(4):pf Taylor}
	2(m+1)(m-1) = (2n-2m+1)(2m-1) -  2m(n-m) - 1.
\end{equation}
Simplify \eqref{eq(4):pf Taylor} to get the equation $(m-1)(n-2m)=0$.
Since $m\neq1$, we have $n=2m$.
Using this equation, we find that the parameters in \eqref{eq:Sm(n)parameters} and \eqref{eq(3):pf Taylor} are equal.
Therefore, $\Delta$ has the same parameters as the Steiner graph $S_m(n)$.
\end{proof}

\begin{example}
(i) The Johnson graph $J(6,3)$ has intersection array $\{9,4,1;1,4,9\}$.
Its local graph is strongly regular with parameters $(9,4,1,2)$ and eigenvalues $4,1,-2$.
Note that $m=2$ and $n=3$.
Thus, every local graph of $J(6,3)$ has the same parameters as the block graph of $\operatorname{OA}(2,3)$. 
Indeed, $J(6,3)$ is locally the block graph of $\operatorname{OA}(2,3)$ since the structure of the local graphs is determined by their parameters.\\
(ii) The halved $6$-cube has intersection array $\{15,6,1;1,6,15\}$.
Its local graph is strongly regular with parameters $(15,8,4,4)$ and eigenvalues $8,2,-2$.
Note that $m=2$ and $n=4$.
Thus, every local graph of the halved $6$-cube has the same parameters as the Steiner graph $S_2(4)$. By the same reason as in (i), the halved $6$-cube is locally the Steiner graph $S_2(4)$.\\
(iii) The Taylor graph from the Kneser graph $K(6,2)$ has intersection array $\{15,8,1;1,8,15\}$.
Its local graph is strongly regular with parameters $(15,6,1,3)$ with eigenvalues $6,1,-3$.
Note that  $m=3$ and $n=4$.
Neither $n=2m-1$ nor $n=2m$ is satisfied. 
Therefore, the Taylor graph from $K(6,2)$ is not locally the block graph of an orthogonal array or a Steiner graph.
\end{example}

%%%%%%%%%%%%%%%%%%%%%%%%%%%%%%%%%%%%%%%%%
%%%%%%%%%%%%%%%%%%%%%%%%%%%%%%%%%%%%%%%%%
%%%%%%%%%%%%%%%%%%%%%%%%%%%%%%%%%%%%%%%%%
% new section
\section{Proof of Conjecture \ref{conj:JV}}\label{Sec:conjecture}

In this section, we consider tight distance-regular graphs with classical parameters and prove Conjecture \ref{conj:JV}. 
We begin by recalling the notion of classical parameters.
For a non-zero integer $b$, we define
\begin{equation*}
	\gauss{i}{1} = \gauss{i}{1}_b := 1 + b + b^2 + \cdots + b^{i-1}.
\end{equation*}
Let $\Gamma$ be a distance-regular graph with diameter $D\geq 3$.
We say $\Gamma$ has \emph{classical parameters} $(D, b, \alpha, \beta)$ whenever its intersection array $\{b_0, b_1, \ldots, b_{D-1}; c_1, c_2, \ldots, c_D\}$ satisfies
\begin{align*}
	&& b_i &= \left( \gauss{D}{1} - \gauss{i}{1} \right)\left( \beta - \alpha\gauss{i}{1} \right) && (0 \leq i \leq D-1),&& \\
	&& c_i & = \gauss{i}{1}\left( 1+\alpha\gauss{i-1}{1} \right) &&  (1\leq i \leq D).
\end{align*}
We note that if $\Gamma$ has classical parameters $(D, b, \alpha, \beta)$, then $\Gamma$ is tight if and only if $\beta=1+\alpha\gauss{D-1}{1}$ and $b,\alpha>0$; see \cite[Proposition 2]{AJJV}.

\begin{lemma}[cf. {\cite[Theorem 7]{AJJV}}]\label{thm7,JV}
Let $\Gamma$ be a tight distance-regular graph with valency $k$, intersection number $a_1$, and classical parameters $(D,b,\alpha, \beta)$. 
Then, its local graphs are strongly regular with parameters $(k,a_1,\lambda,\mu)$, where
$$
	\mu = \alpha(b+1), \qquad \lambda=(\alpha-1)(b+1)+\alpha b \gauss{D-2}{1},
$$
and eigenvalues $a_1>r>s$, where
\begin{equation}\label{eq:tight local eig classical para}
	a_1=\alpha(b+1)\gauss{D-1}{1}, \qquad
	r = \alpha b \gauss{D-2}{1}, \qquad
	s = -1-b.
\end{equation}
\end{lemma}

\begin{remark}\label{rmk:b=b1/(1+th1)} Let $\Gamma$ be a tight distance-regular graph with classical parameters $(D,b,\alpha, \beta)$ and smallest eigenvalue $s$. 
From the equations $s=-1-b_1/(1+\theta_1)$ in \eqref{eq:local eig tight} and $s = -1-b$ in \eqref{eq:tight local eig classical para}, $\Gamma$ satisfies 
\begin{equation}
	b = \frac{b_1}{1+\theta_1}.
\end{equation}
\end{remark}

Now, we are ready to prove Conjecture \ref{conj:JV}. 

\begin{theorem}[cf. {\cite[Conjecture 2]{AJJV}}]\label{main thm2}
Let $\Gamma$ be a tight distance-regular graph with classical parameters $(D,b,\alpha, \beta)$, where $D\geq 3$ and $b\geq 2$. Then, a local graph of $\Gamma$ is neither the block graph of an orthogonal array or a Steiner system.
\end{theorem}

\begin{proof}
For a vertex $x \in V(\Gamma)$, let $\Delta$ denote the local graph of $\Gamma$ at $x$.
Since $\Gamma$ is tight and by Lemma \ref{thm7,JV}, $\Delta$ is a strongly regular graph with eigenvalues $a_1, r, s$ from \eqref{eq:tight local eig classical para}.
From Remark \ref{rmk:b=b1/(1+th1)}, $\Gamma$ satisfies that $b=b_1/(1+\theta_1)$.
Set $m:=-s=1+b$ and $n:=r-s = \alpha b\gauss{D-2}{1}+1+b$.
Observe that $n>m$ and $\Delta$ has the smallest eigenvalue $-m$ with $m\geq 3$.
Now, we consider two cases: (i) $\Delta$ is the block graph of an orthogonal array; (ii) $\Delta$ is the block graph of a Steiner system.

\noindent
\textbf{Case (i):} Suppose $\Delta$ is the block graph of an orthogonal array.
Consider the parameter $\mu$ of $\Delta$.
By Lemma \ref{prop: XOA(m,n)} we have $\mu =m(m-1)$ and by Lemma \ref{thm7,JV} we have $\mu = \alpha(1+b)$. 
By these comments and since $m=1+b$, it follows $\alpha=b$.
Thus, the intersection number $c_2$ of $\Gamma$ is given by
	$$
		c_2 = \gauss{2}{1}\left(1+\alpha\gauss{1}{1}\right) = (1+b)(1+\alpha) = (1+b)^2.
	$$
However, this contradicts the result of Corollary \ref{cor:sub-main}(i).

\noindent
\textbf{Case (ii):} The argument is similar to Case (i). 
Suppose $\Delta$ is the block graph of a Steiner system $S(2,m,n)$.
%{\color{red} \sout{By Lemma \ref{SRG Thm8.6.4}, $\Delta$ is the Steiner graph $S_m(n)$ with $n>m$.} }
Consider the parameter $\mu$ of $\Delta$.
By Lemma \ref{prop:blockgraphdesign} and Lemma \ref{thm7,JV}, we have $\mu=m^2=\alpha(b+1)$.
Since $m=b+1$, it follows $\alpha=b+1$.
Thus, the intersection number $c_2$ of $\Gamma$ is given by
	$$
		c_2 = \gauss{2}{1}\left(1+\alpha \gauss{1}{1}\right) = (1+b)(1+\alpha) = (1+b)(2+b).
	$$
However, this contradicts the result of Corollary \ref{cor:sub-main}(ii).

\noindent
Consequently, $\Delta$ is neither the block graph of an orthogonal array nor the block graph of a Steiner system.
The result follows.
\end{proof}

%%%%%%%%%%%%%%%%%%%%%%%%%%%%%%%%%%%%%%%%%%%%%%%%%%%%%%%%%%%%%%%%%%%%%%%%%%%%%%%%%%%%%%%%%%%%%
%%%%%%%%%%%%%%%%%%%%%%%%%%%%%%%%%%%%%%%%%%%%%%%%%%%%%%%%%%%%%%%%%%%%%%%%%%%%%%%%%%%%%%%%%%%%%
\section{Proof of Theorem \ref{main thm3}}\label{Sec:main thm3}

In this section, we prove Theorem \ref{main thm3}.
To do this, we recall some known results that we need in the proof.

\begin{lemma}[cf. {\cite[Theorem 3.1]{Neumaier1979}}]\label{lem:mu-bound}
Let $\Gamma$ be a primitive strongly regular graph with parameters $(v,k,\lambda,\mu)$ and integral eigenvalues $k>r>s=-m$. Then 
\begin{equation}\label{eq:mu-bound}
	\mu \leq m^3(2m-3).
\end{equation} 
If equality holds, then $n=m(m-1)(2m-1)$, where $n=r-s$.
\end{lemma}

\begin{lemma}[cf. {\cite[Theorem 8.6.3]{SRG}}]\label{SRG Thm8.6.3}
Let $\Gamma$ be a primitive strongly regular graph with parameters $(v,k,\lambda,\mu)$ and integral eigenvalues $k>r>s$.
For convenience, we set $m:=-s$ and $n:=r-s$.
Let $f(m,\mu)=\frac{1}{2}m(m-1)(\mu+1)+m-1$. Then
\begin{enumerate}[\normalfont(i)]
	\item If $\mu=m(m-1)$ and $n>f(m,\mu)$, then $\Gamma$ is the block graph of an orthogonal array $\operatorname{OA}(m,n)$.
	\item If $\mu=m^2$ and $n>f(m,\mu)$, then $\Gamma$ is the block graph of a Steiner system $S(2,m,mn+m-n)$.
	\item {\rm(Claw bound)} If $\mu \ne m(m-1)$ and $\mu\ne m^2$, then $n\leq f(m,\mu)$.
\end{enumerate}
\end{lemma}

Now we prove Theorem \ref{main thm3}.

\begin{proof}[Proof of Theorem \ref{main thm3}]
Let $\Delta$ denote a local graph of $\Gamma$ at a vertex $x\in V(\Gamma)$.
Then $\Delta$ is strongly regular with parameters $(k, a_1, \lambda, \mu)$ and eigenvalues $a_1, r, s$ from \eqref{eq:local eig tight}.
Set $m:=-s$ and $n:=r-s$.
By the given condition, $\Delta$ is neither the block graph of an orthogonal array nor the block graph of a Steiner system.
By Lemma \ref{SRG Thm8.6.3}, we find
\begin{equation}\label{claw bound}
	n \leq \frac{1}{2} m(m-1)(\mu+1) + m-1.
\end{equation}
Substitute $n=r+m$ into \eqref{claw bound} and simplify the result to obtain
\begin{equation}\label{claw bound;r,m,mu}
	r \leq \frac{1}{2} m(m-1)(\mu+1) -1.
\end{equation}
Apply \eqref{eq:mu-bound} to \eqref{claw bound;r,m,mu} to obtain
\begin{equation}\label{claw bound;r,m,mu(2)}
	r \leq \frac{1}{2} m(m-1)(m^3(2m-3)+1) -1.
\end{equation}

\smallskip
\noindent
Next, we recall the equation $\mu=a_1+rs$ from \eqref{eq:srg eigval}.
Eliminate $\mu$ in \eqref{eq:mu-bound} using this equation and simplify the result using $s=-m$ to obtain
\begin{equation}\label{a1 bound(1)}
	a_1 \leq m^3(2m-3) + rm.
\end{equation}
Eliminate $r$ in the right-hand side of \eqref{a1 bound(1)} by applying the inequality \eqref{claw bound;r,m,mu(2)} and then simplify the result to obtain
\begin{equation}\label{a1 bound(2)}
	a_1 \leq g(m),
\end{equation}
where $g(m)=\frac{1}{2}\big(m^3(2m-3)+1\big)\big(m^2(m-1)+2\big)-m-1$.
We note that $a_1$ is the valency of $\Delta$ and the diameter of $\Delta$ is two.
Thus, by \eqref{MooreBound} we have
\begin{equation}\label{eq:Moore bound}
	|V(\Delta)|=k \leq 1+a_1^2.
\end{equation}
Applying the inequality \eqref{a1 bound(2)} to the right-hand side of \eqref{eq:Moore bound}, we find
\begin{equation*}
	k \leq 1+g(m)^2.
\end{equation*}
Since $m=1+b$, the valency $k$ of $\Gamma$ is bounded by a function in $b$.
Since the diameter of a distance-regular graph is bounded in terms of its valency (cf.~\cite[Section~4]{BDKM2015}), we conclude that the diameter of $\Gamma$ is bounded by a function in $b$.
The result follows.
%{\color{red} \sout{In the results of the Bannai-Ito conjecture, it is known that the diameter of a distance-regular graph is bounded in terms of its valency. Consequently, the diameter of $\Gamma$ is bounded by a function in $b$.} }
\end{proof}

\begin{remark}\label{rmk:bound for k}
Referring to the proof of Theorem \ref{main thm3}, the valency $k$ is bounded by a function $\varphi$ in the variable $b$, where
\begin{equation*}
	\varphi(b) = \frac{1}{4}\left[ \left((1+b)^3(2b-1)+1\right)\left( b(1+b)^2+2 \right) - 2b - 4 \right]^2+1.
\end{equation*}
Since $b=m-1$, we also find that the diameter of $\Gamma$ is bounded by a function in the variable $m$, where $-m$ is the smallest eigenvalue of a local graph of $\Gamma$.
\end{remark}

\begin{corollary}\label{cor:t-DRG,D,b}
Let $\Gamma$ be a tight distance-regular graph with classical parameters $(D,b,\alpha, \beta)$, $D\geq 3$, $b\geq 2$.
Then, the diameter of $\Gamma$ is bounded by a function in $b$. 
\end{corollary}
\begin{proof}
Let $k>\theta_1>\ldots> \theta_D$ be eigenvalues of $\Gamma$.
From Remark \ref{rmk:b=b1/(1+th1)}, $\Gamma$ satisfies that $b=b_1/(1+\theta_1)$.
By Theorem \ref{main thm2}, a local graph of $\Gamma$ is neither the block graph of an orthogonal array nor the block graph of a Steiner system.
Therefore, by Theorem \ref{main thm3}, the diameter of $\Gamma$ is bounded by a function in $b$.
The result follows.
\end{proof}

We conclude the paper with a brief summary and a discussion of further direction.
We considered a distance-regular graph $\Gamma$ with diameter $D\geq 3$. 
Assuming that $\Gamma$ is locally strongly regular with smallest eigenvalue $-m$, where $m\geq 3$, and the intersection number $\gamma(\Gamma)$ exists, we have shown our main result that if $\Gamma$ is locally the block graph of an orthogonal array (resp. a Steiner system), then the intersection number $c_2$ is not equal to $m^2$ (resp. $m(m+1)$).
In particular, when $\Gamma$ is tight with classical parameters, it is not locally the block graph of an orthogonal array or a Steiner system.
Additionally, using the main result, we have proven that if $\Gamma$ is tight and not locally the block graph of an orthogonal array or a Steiner system, then the diameter of $\Gamma$ is bounded by a function of the parameter $b=b_1/(1+\theta_1)$.
As we mentioned in Section \ref{Sec:Intro}, it is a significant problem to determine an upper bound for the diameter of distance-regular graphs using some intersection numbers of $\Gamma$.
Our future goal is to generalize Theorem \ref{main thm3}, demonstrating that the diameter of tight distance-regular graphs is bounded by a function of the variable $b$.
We present the following conjecture.
\begin{conjecture}\label{conj:D < f(b)}
Let $\Gamma$ be a tight distance-regular graph.
Let $b=b_1/(1+\theta_1)$, where $b_1$ is the intersection number of $\Gamma$ and $\theta_1$ is the second largest eigenvalue of $\Gamma$, and assume $b \geq 2$.
Then, the diameter of  $\Gamma$ is bounded by a function in $b$.
\end{conjecture}

\begin{remark}
To prove Conjecture \ref{conj:D < f(b)}, according to Theorem \ref{main thm3}, it suffices to prove that for tight distance-regular graphs with $D\geq 3$ which are locally the block graphs of orthogonal arrays or Steiner systems, their diameters are bounded by a function in $b$, provided $b\geq 2$. 
Furthermore, it is worth noting that, except for the halved $2D$-cubes and the Johnson graphs $J(2D,D)$, all known tight distance-regular graphs have diameter at most $4$.
\end{remark}

\section*{Acknowledgements}
The authors would like to express our sincere thanks to the referees for their valuable comments and suggestions, which helped enhance the quality of our paper.
J.H. Koolen is partially supported by the National Key R. and D. Program of China (No. 2020YFA0713100), the National Natural Science Foundation of China (No. 12071454 and No. 12371339), and the Anhui Initiative in Quantum Information Technologies (No. AHY150000). Shuang-Dong Li is supported by Natural Science Research Project of Anhui Educational Committee (No. 2023AH053262). Xiaoye Liang is partially supported by the National Natural Science Foundation of China (No. 12201008), the Foundation of Anhui Jianzhu University (No. 2022QDZ18).  Ying-Ying Tan is partially supported by the National Natural Science Foundation of China (No. 12371339 and No. 12171002), Natural Science Research Project of Anhui Educational Committee (No. 2023AH050194), the Innovation Team of Operation Research and Combinatorial Optimization of Anhui Province (No. 2023AH010020).


\begin{thebibliography}{XX}


%\bibitem{BBIT}
%E. Bannai, E. Bannai, T. Ito, R. Tanaka,
%Algebraic Combinatorics,
%Boston: De Gruyter, Berlin, 2021.

\bibitem{BI1984}
E. Bannai, T. Ito,
Algebraic Combinatorics I: Association Schemes,
Benjamin/Cummings, London, 1984.

\bibitem{BDKM2015}
S. Bang, A. Dubickas, J.H. Koolen, V. Moulton,
There are only finitely many distance-regular graphs of fixed valency greater than two,
Adv. Math. {269} (2015), 1--55.

\bibitem{bcn89}
A.E. Brouwer, A.M. Cohen, A. Neumaier,
Distance-Regular Graphs, Springer-Verlag, Berlin, 1989.

\bibitem{SRG}
A.E. Brouwer, H. Van Maldeghem,
Strongly Regular Graphs,
Cambridge University Press, Cambridge, 2022.


%\bibitem{HandbookCD}
% C.J. Colbourn, J.H. Dinitz,
% Handbook of combinatorial designs,
% Chapman \& Hall, London, 2007.

\bibitem{EKR}
C. Godsil, K. Meagher,
Erd\H{o}s-Ko-Rado Theorems: Algebraic Approaches,
Cambridge University Press, Cambridge, 2016.

%
%\bibitem{Bang:Delsarte}
%S. Bang, A. Hiraki, J.H. Koolen, Delsarte clique graphs, European J. Combin., 28 (2007) 501-516.
%
%\bibitem{Bang}
%S. Bang, T. Fujisaki, J.H. Koolen, The spectra of the local graphs of the twisted Grassmann graphs, European J. Combin., 30 (2009) 638-654.
%
% \bibitem{BH}
% A.E. Brouwer, W.H. Haemers, \emph{Spectra of Graphs}, Springer-Verlag, Heidelberg, 2012.
%

%
%%\bibitem{vandam}E.R. van Dam, Regular graphs with four eigenvalues. Linear Algebra Appl., 226-228 (1995) 139-162.
%
%\bibitem{vandamKoolen}
%E.R. van Dam, J.H. Koolen, A new family of distance-regular graphs with unbounded diameter, Invent. Math., 162 (2005) 189-193.
%
\bibitem{vDKT2016}
E.R. van Dam, J.H. Koolen, H. Tanaka,
Distance-regular graphs,
Electron. J. Comb., Dynamic Survey (2016) \#DS22.

\bibitem{GoTer2002}
J.T. Go, P. Terwilliger,
Tight distance-regular graphs and the subconstituent algebra,
European J. Combin.  23 (2002) 793--816.

%%\bibitem{Haemers}W.H. Haemers, Interlacing eigenvalues and graphs, Linear Algebra Appl., 226-228 (1995) 539-616.
%
%\bibitem{GavrilyukGrassmann}
% A.L.~Gavrilyuk, J.H. Koolen, On a characterization of Grassmann graphs, https://arxiv.org/abs/1806.02652v1.
%
%\bibitem{GavrilyukTerpoly}
% A.L. Gavrilyuk, J.H. Koolen, The Terwilliger polynomial of a Q-polynomial distance-regular graph and its application to pseudo-partition graphs, Linear Algebra Appl., 466 (2015) 117-140.
%
%\bibitem{GodsilEKR}
% C.D. Godsil, K. Meagher, \emph{Erd\H{o}s-Ko-Rado Theorems: Algebraic Approaches}, Cambridge Studies in Advanced Mathematics 149, 2016.
%
%\bibitem{Godsil}
% C.D. Godsil, G. Royle, \emph{Algebraic Graph Theory}, Springer-Verlag, Berlin, 2001.
%
%\bibitem{Hayat}
% S. Hayat, J.H. Koolen, M. Riaz, A spectral characterization of the s-clique extension of the square grid graphs, European  J. Combin., 76 (2019) 104-116.


\bibitem{AJJKPT}
A. Juri{\v s}i{\'c}, J.H. Koolen, P. Terwilliger,
Tight distance-regular graphs,
J. Algebraic Comb., {12} (2000) 163--197.

\bibitem{2002JKDiscMath}
A. Juri\v{s}i\'{c}, J.H. Koolen,
Krein parameters and antipodal tight graphs with diameter $3$ and $4$,
{Discrete Math.}, {244} (2002) 181--202.

\bibitem{JMT}
A. Juri{\v s}i{\'c}, A. Munemasa, Y. Tagami,
On graphs with complete multipartite $\mu$-graphs,
Discrete Math., {310} (2010) 1812--1819.

\bibitem{AJJV}
A. Juri{\v s}i{\'c}, J. Vidali,
Restrictions on classical distance-regular graphs,
J. Algebraic Comb., {46} (2017) 571--588.


\bibitem{KAGL2024+}
J.H. Koolen, M. Abdullah, B. Gebremichel, J.-H. Lee,
Towards a classification of $1$-homogeneous distance-regular graphs with positive intersection number $a_1$, \textit{submitted}.


\bibitem{KP2012}
J.H. Koolen, J. Park,
Distance-regular graphs with or at least half the valency,
J. Comb. Theory, Ser. A, {119} (3) (2012), 546--555.


%\bibitem{MartinTanaka}
%W.J. Martin, H. Tanaka,
%Commutative association schemes,
%European J. Combin., {30} (2009) 1497--1525.

\bibitem{MillerSiran2013}
M. Miller, J. \v{S}ir\'a\v{n},
Moore graphs and beyond: A survey of the degree/diameter problem,
Electron. J. Comb. 20(2) (2013) \#DS14v2.

 \bibitem{Neumaier1979}
 A. Neumaier,
 Strongly regular graphs with smallest eigenvalue $-m$,
 Arch. Math.  {33} (1979) 392--400.

% \bibitem{NeuPen2022}
% A. Neumaier, S. Penji\'{c},
% On bounding the diameter of a distance-regular graph,
% Combinatorica {42}(2) (2022) 237--251.

% \bibitem{Nomura1994}
% K. Nomura,
% Homogeneous graphs and regular near polygons,
% J. Combin. Theory Ser. B 60 (1994) 63--71.

 \bibitem{Pasc2001}
 A. Pascasio,
 Tight distance-regular graphs and the $Q$-polynomial property,
Graphs Combin.  17 (2001) 149--169.

\bibitem{Pasc2001LAA}
 A.  Pascasio,
 An inequality on the cosines of a tight distance-regular graph,
 Linear Algebr. Appl., 325 (2001) 147--159.

% \bibitem{LeeKoolenTan}
% J.H. Koolen, J-H Lee, Y-Y Tan, Remarks on pseudo-vertex-transitive graphs with small diameter, https://arxiv.org/abs/2102.00105.
%
% \bibitem{Metsch}
%K. Metsch, A characterization of Grassmann graphs, European J. Combin., 16 (1995) 639-644.
%
% \bibitem{Ray}
%D.K. Ray-Chaudhuri, A.P. Sprague, Characterization of projective incidence structures, Geom. Ded., 5 (1976) 361-376.
%
% \bibitem{Ternote}
%P.~Terwilliger, Algebraic Graph Theory, Lecture notes, Unpublished, See the link https://icu-hsuzuki.github.io/lecturenote/.
%
%\bibitem{Ter1}
%P. Terwilliger, The subconstituent algebra of an association scheme I, J. Algebraic Combin., 1 (1992) 363-388.

%\bibitem{Ter3}P. Terwilliger, The subconstituent algebra of an association scheme III, J. Algebraic Combin., 2 (1993) 177-210.

\end{thebibliography}
\end{document}